\documentclass[oneside,a4paper,11pt]{article}
\usepackage{mathtools}
\usepackage{amsmath}
\allowdisplaybreaks[4]
\usepackage{csquotes}
\usepackage{amsthm}
\usepackage{amssymb}
\usepackage{mathrsfs}
\usepackage{color}
\usepackage{bm}
\usepackage{fancyhdr}
\usepackage{geometry}
\usepackage{hyperref}
\usepackage{enumerate}
\usepackage{graphicx}
\usepackage{subfig}
\usepackage{float}

\theoremstyle{definition}
\newtheorem{thm}{Theorem}[section]

\newtheorem{lem}[thm]{Lemma}
\newtheorem{prop}[thm]{Proposition}
\newtheorem{defn}[thm]{Definition}

\newtheorem{ex}[thm]{Example}
\numberwithin{equation}{section}
\newcommand{\ddt}[1]{\frac{\mathrm{d}#1}{\mathrm{d}t}}
\title{A combinatorial curvature flow in spherical background geometry}
\author{Huabin Ge~~
	Bobo Hua~~
	Puchun Zhou
}

\date{}
\usepackage{fancyhdr}
\providecommand{\classification}[1]
{
	\small	
	\textbf{Mathematics Subject Classification (2020):} #1
}

\begin{document}
	\maketitle
	\begin{abstract}
	In \cite{ge2021combinatorial}, the existence of ideal circle patterns in Euclidean or hyperbolic background geometry under the combinatorial conditions was proved using flow approaches. It remains as an open problem for the spherical case. In this paper, we introduce a combinatorial geodesic curvature flow in spherical background geometry, which is analogous to the combinatorial Ricci flow of Chow and Luo in \cite{MR2015261}. 
	We characterize the sufficient and necessary condition for the convergence of the flow. That is, the prescribed geodesic curvature satisfies certain geometric and combinatorial condition if and only if for any initial data the flow converges exponentially fast to a circle pattern with given total geodesic curvature on each circle. Our result could be regarded as a resolution of the problem in the spherical case. As far as we know, this is the first combinatorial curvature flow in spherical background geometry with fine properties, and it provides an algorithm to find the desired ideal circle pattern. 	\\
		\classification{52C26, 51M10, 57M50
		}	
		
	\end{abstract}
	
	\section{Introduction} 
 		Circle patterns are used for constructing hyperbolic 3-manifolds by Thurston \cite{thurston1976geometry}. The existence and uniqueness of certain type of circle patterns on surfaces are known as the Koebe-Andreev-Thurston theorem, which has many different proofs in the literature. Colin de Verdi\`eres \cite{MR1106755} proposed a variational principle, and proved that the existence of circle patterns corresponds to the existence of minimizers of some potential function. See \cite{MR2022715,guo2007note} for related results.
	 
		Inspired by Ricci flows on Riemannian manifolds, Chow and Luo introduced the combinatorial Ricci flow in \cite{MR2015261}
		\begin{align*}
			\ddt{r_i} = -K_is(r_i),
		\end{align*}
		where $s(r)=r$ ($\sin r$ or $\sinh r$) for Euclidean (spherical or hyperbolic) background geometry, and $K_i$ is the discrete Gaussian curvature given by the difference of $2\pi$ and the cone angle at the center of the circle $i.$ They proved that the combinatorial Ricci flow on a surface of genus at least one converges to a smooth Euclidean or hyperbolic metric under certain combinatorial conditions. 
After that, there are lots of works on combinatorial curvature flows; see e.g. \cite{MR2100762,MR2136535,gu2008computational,MR3818085,MR3825607,MR3807319,ge2018combinatorial,luo2019koebe,ge20203,ge2021combinatorial,ge2022deformation,feng2022combinatorial2}.

		However, there are less results on circle patterns in spherical background geometry. Main difficulties are as follows: related potential functions are not convex in the spherical case (see e.g. \cite{MR3178441}), and the uniqueness of circle patterns fails due to the invariance under the M\"obius transformation. 
		Recently, Nie \cite{nie2023circle} proposes a new potential function, which is convex for circle patterns with conic singularities in the spherical case. In this paper, we introduce a combinatorial prescribed geodesic curvature flow, and prove the convergence of the flow to the desired circle pattern.
		
		\subsection{Spherical circle patterns on surfaces}
		We introduce the setting of circle patterns. Let $(V,E)$ be a graph and $\Sigma$ a closed surface. Let $\eta: V\cup E\to \Sigma$ be a graph embedding. A face is a connected component of $\Sigma\setminus \eta(V\cup E),$ and $F$ is the set of faces induced by $\eta$. We will not distinguish vertices (or edges) with their images via $\eta.$ We call an embedding $\eta$ a closed 2-cell embedding if the following hold; see e.g. \cite{barnette1987generating}:
		\begin{enumerate}
			\item The closure of every face is homeomorphic to a closed disk.
			\item Any face is bounded by a simple closed curve consists of finite many edges.
		\end{enumerate}
		 Let $G=(V,E,F)$ be a closed 2-cell embedding in a closed surface $\Sigma$ with vertex set $V$, edge set $E$ and face set $F$.
		 We write $v<e$ ($e<f$ resp.) if a vertex $v$ (an edge $e$ resp.) is incident to an edge $e$ (a face $f$ resp.). Besides, we write $v_1\sim v_2$  if vertices $v_1$ and $v_2$ are linked by an edge. For a set of vertices $X\subseteq V$, we denote by $E(X)$ the edges which are incident to some vertices in $X$. For any face $f$, we add an auxiliary vertex $v_f$ in the interior of the face, labelled as a small triangle in Figure \ref{fig01}. We denote by $V_F$ the set of those auxiliary vertices. We define the incidence graph of $G$; see e.g. \cite{coxeter1950self}.

		\begin{figure}[H]
		\centering
		\includegraphics[width=5in]{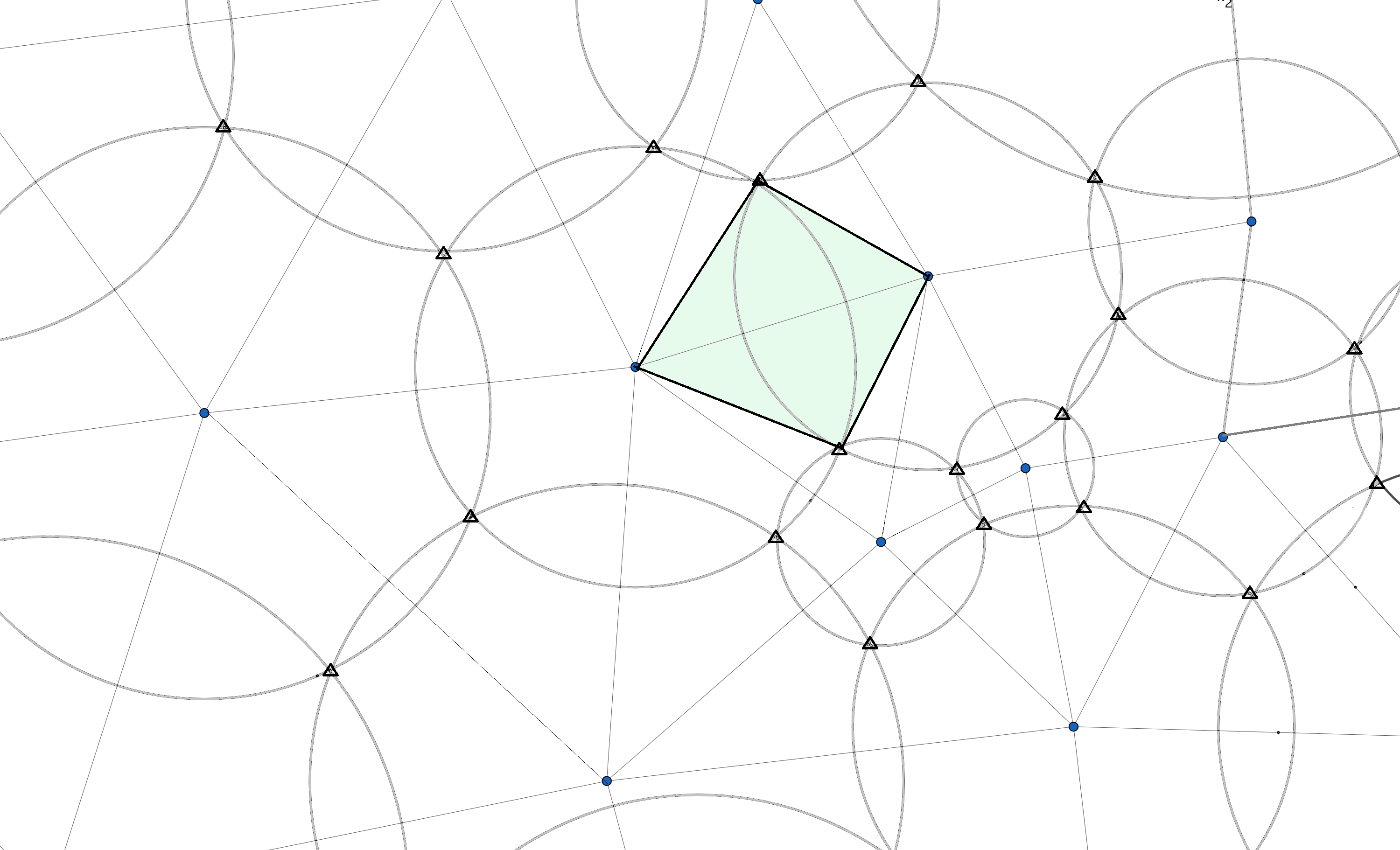}
		\caption{The picture of a circle pattern}
		\label{fig01}
	\end{figure}

	 \begin{defn}
	An incidence graph $I(G)$ is a bipartite graph with the bipartition $\{V,V_F\}$. For $v\in V$ and $v_f\in V_F$, $v$ and $v_f$ are adjacent in $I(G)$ if and only if $v$ is on the boundary of the face $f.$
\end{defn}
	
	Note that the incidence graph $I(G)$ induces a closed 2-cell embedding with quadrilateral faces.
For each edge in $G,$ it associates with a unique quadrilateral in $I(G)$ containing it. We assign a radius $r_v\in(0,\frac{\pi}{2})$ for each $v\in V$, and an intersection angle $\Theta(e)\in (0,\frac{\pi}{2}]$ for each edge $e\in E.$ We define a piecewise spherical metric for $\Sigma$ via the radii and  intersection angles.	
		 For each edge $e_0=\{v_1,v_2\}\in E,$ we write $A$ and $B$ for the remaining two vertices of the quadrilateral associated with $e_0$ as in Figure \ref{fig02}. 
		 We endow the quadrilateral $v_1Av_2B$ with the metric of a spherical quadrilateral with $|v_1A|=|v_1B|=r_{v_1}$, $|v_2A|=|v_2B|=r_{v_2}$ and angles $\angle v_1Av_2=\angle v_1Bv_2=\pi-\Theta(e_0)$ via the intersection of two spherical disks. This spherical quadrilateral is called a spherical bigon determined by $r_{v_1},r_{v_2}$ and $\Theta(e_0).$
		We obtain a metric structure on $\Sigma,$ denoted by $S(\Sigma),$ via gluing those quadrilaterals (spherical bigons) along common edges in $I(G)$, see the gluing procedure in \cite[Chapter 3]{burago2022course}, which is a piecewise spherical metric with possible conic singularities.

	We denote by $C_v$ the circumcircle with center $v$ and radius $r_v$ in $S(\Sigma).$ Those circumcircles form a circle pattern $\mathcal{C}$ on $S(\Sigma)$ as in Figure \ref{fig01}. This is called the ideal circle pattern in the literature; see \cite{MR2022715} and Definition~1.1 in \cite{ge2021combinatorial}. Possible conic singularities of $S(\Sigma)$ only appear in $V$ and $V_F.$ We denote by $\alpha_i$ the cone angle at the point $v_i\in V$, which equals to the sum of angles of quadrilaterals at $v_i$. Besides, the cone angle at $v_f\in V_F$ equals to $\sum_{e<f}(\pi-\Theta(e))$. If all the cone angles equal to $2\pi$, we obtain a surface with smooth spherical metric. By $k_v$ we denote the geodesic curvature on $C_v$, which equals to $\cot r_v$. The length of $C_v,$ denoted by $l_v,$ is $\alpha_v\sin r_v.$ Hence the total geodesic curvature of $C_v,$ denoted by $L_v,$ is given by $l_vk_v=\alpha_v\cos r_v$.

   \begin{figure}[H]
	\centering
	\includegraphics[width=4in]{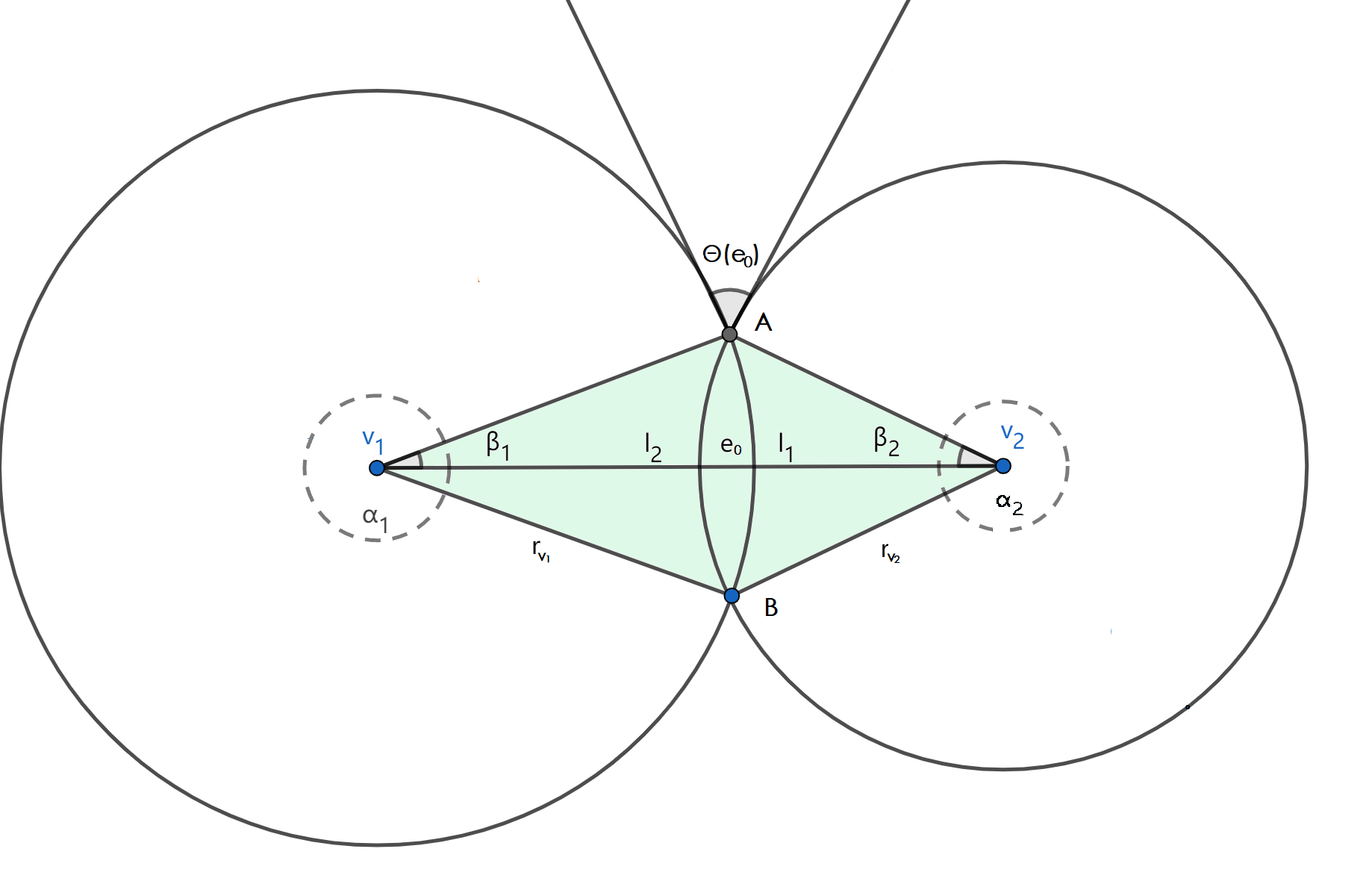}
	\caption{Two intersecting circles and the quadrilateral.}
	\label{fig02}
\end{figure}

	\subsection{Prescribed geodesic curvature flows}
	The aim of our article is to study the prescribed geodesic curvature problem of circle patterns. Given intersection angles $\{\Theta(e)\}_{e\in E}$ and prescribed total geodesic curvatures $\{\hat{L}_v\}_{v\in V}$, we want to find radii $\{r_v\}_{v\in V}$ such that the total geodesic curvature of each $C_v$ in $S(\Sigma)$ satisfies $$L_v=\hat{L}_v,\quad \forall v\in V.$$ To solve this problem, we introduce the following prescribed geodesic curvature flow
	    \begin{align}
		\ddt{r_v}=\frac{(L_v-\hat{L}_v)}{2}\sin(2r_v),\quad \forall v\in V\label{flow}.
		\end{align}
		We prove the long time existence and the uniqueness of the flow in Theorem~\ref{thm:longtime}.
We say that the flow converges if there is $r_v^*\in(0,\frac{\pi}{2})$ for each $v\in V$, such that $r_v(t)\rightarrow r_v^*$ as $t\rightarrow\infty$. 
For the prescribed geodesic curvature flow, we have the main result.
\begin{thm}\label{thm1}
	Let $G=(V,E,F)$ be a closed 2-cell embedding in a closed surface $\Sigma.$ Let $\Theta(e)\in (0,\frac{\pi}{2}),$ $e\in E,$ be intersection angles, and $\hat{L}_v>0$ for $v\in V.$ The following are equivalent:
\begin{enumerate}[i.]
	\item The prescribed geodesic curvature flow \eqref{flow} with $\{\hat{L}_v\}_{v\in V}$  converges for any initial data.
	\item The prescribed total geodesic curvature $\{\hat{L}_v\}_{v\in V}$ satisfies
	\begin{align}
		\sum_{v\in X}\hat{L}_v<2\sum_{e\in E(X)}\Theta(e),\quad\quad \forall X\subseteq V\label{condition}.
	\end{align}
\end{enumerate}
	Moreover, if the flow \eqref{flow} converges, then it converges exponentially fast to a unique circle pattern with $L_v=\hat{L}_v$ for each vertex $v$.
\end{thm}
The proof strategy is as follows: we observe that the prescribed geodesic curvature flow \eqref{flow}
 is a negative gradient flow of Nie's potential function \eqref{eq:nie}. By \eqref{condition}, the function is
 a proper strictly convex function. Using the theory of gradient flows, we prove the desired result.
 Note that the result of exponential convergence helps for developing an algorithm for the numerical computation of the metric.
 
 The paper is organized as follows: in next section, we recall the variational principle for the prescribed geodesic curvature flow. Section~\ref{sec:3} is devoted to the proof of Theorem \ref{thm1}.
   
\section{Variational principle}	
	The variational principle of the total geodesic curvature in spherical background geometry was introduced in \cite{nie2023circle}. 
	 In the rest of the paper, we write $V=\{v_i\}_{i=1}^N$. Consider a spherical bigon associated with an edge $e_0,$ see Figure \ref{fig02}.
	Let $l_i(i=1,2)$ denote the length of the arc of $C_{v_i} (i=1,2),$ which is contained in the quadrilateral (spherical bigon). $L_i$ is the total geodesic curvature on that arc.
	 By Gauss-Bonnet theorem, one obtains that \begin{align}
	 	A=2\Theta(e_0)-L_1-L_2 \label{gauss},
	 \end{align} 
 where $A$ is the area of the intersection part of two spherical disks. Let $K_i$ denote $\ln k_{v_i}$. Then we have the key variational formula.
	\begin{lem}\label{vari} For a spherical bigon,
		$l_1dk_{v_1}+l_2dk_{v_2}=L_1dK_1+L_2dK_2$ is a closed form. Moreover, we have 
		\begin{align}
			\frac{\partial L_1}{\partial K_2}<0,
			\frac{\partial (L_1+L_2)}{\partial K_1}>0.
		\end{align}
	\end{lem}
	Intuitively, if the circle $C_{v_1}$ is fixed and the radius of $C_{v_2}$ increases, both the area of intersection part and the length $l_1$ increase. The detailed proof can be found in the appendix. We write $$\omega(e_0):=L_1dK_1+L_2dK_2,\quad (K_1,K_2)\in\mathbb{R}^{2}.$$  There is a convex potential function $$\mathcal{E}_{e_0}(K_1,K_2)=\int^{(K_1,K_2)}\omega(e_0).$$
 
 Similarly, we define a potential function $\mathcal{E}_e$ for each edge $e$.
 For $\{\hat{L}_v\}_{v\in V}$, the following potential function was introduced by Nie \cite{nie2023circle}
 \begin{equation}\label{eq:nie}\mathcal{E}(K)=\sum_{e=\{u,w\}\in E}\mathcal{E}_e(K_u,K_w)-\sum_{v\in V}\hat{L}_vK_v.\end{equation}

Now we consider the prescribed geodesic curvature flow.
By the change of variables $K_v=\ln \cot r_v,$ the flow \eqref{flow} is reformulated as

\begin{align}
	\ddt{K_v} = -(L_v-\hat{L}_v)\label{2.2}.
\end{align} One verifies that
the flow \eqref{2.2} is a negative gradient flow of the convex function \eqref{eq:nie}. Moreover, if it has a critical point, it provides the circle pattern with prescribed total geodesic curvatures.

\section{The proof of the main theorem}\label{sec:3}
In this section, we prove the main result, Theorem \ref{thm1}.

For $V=\{v_i\}_{i=1}^N,$ we write 
$$r=(r_{v_1},\cdots, r_{v_N})=(r_{1},\cdots, r_{N}),\quad K=(K_{v_1},\cdots, K_{v_N})=(K_{1},\cdots, K_{N}),$$ where $K_{v_i}=\ln\cot r_{v_i}$ for any $v_i.$ We consider the domain of the variable $r\in (0,\frac{\pi}{2})^N,$ which is 1-1 corresponding to $K\in \mathbb{R}^N.$

We first prove the long time existence and uniqueness of the geodesic curvature flow.
\begin{thm}\label{thm:longtime}
Given $\hat{L}_v>0$ for any $v\in V,$ for any initial data $r_0\in (0,\frac{\pi}{2})^N$ ($K_0\in \mathbb{R}^N$ resp.), the flow \eqref{flow} (\eqref{2.2} resp.) has a unique solution $r:[0,\infty)\to (0,\frac{\pi}{2})^N$ ($K:[0,\infty)\to \mathbb{R}^N$ resp.).
\end{thm}
\begin{proof} Since the flow \eqref{flow} and \eqref{2.2} are equivalent by the change of variables, it suffices to consider the flow \eqref{2.2}. Note that $L_v(K)$ is a smooth function of $K.$ By Picard's theorem, we have the local existence and uniqueness of the flow for any initial data. By \eqref{gauss}, the total geodesic curvature $L_v$ of the circle $C_v$ is less than $\sum_{v<e}2\Theta(e)$ for each vertex $v$. Hence $|L_v(t)-\hat{L}_v|$  is bounded by $\sum_{v<e}2\Theta(e)+\hat{L}_v$. Therefore, the flow \eqref{2.2} exists for $t\in[0,\infty)$.
\end{proof}

Next, we recall the result of Nie in \cite{nie2023circle}.
\begin{thm}[Nie]\label{Nie}
Given a closed 2-cell embedding $G=(V,E,F)$ in a closed surface $\Sigma$ with intersection angles $\{\Theta(e)\}_{e\in E}$ and the prescribed total geodesic curvatures $\{\hat{L}_v\}_{v\in V}$, the following statements hold: 
\begin{enumerate}
	\item $K=(K_{v_1},...,K_{v_{N}})$ is a critical point of $\mathcal{E},$ defined in \eqref{eq:nie}, if and only if there is a circle pattern on $\Sigma$ whose radii are $\{r_v\}_{v\in V}$ with $K_v = \ln\cot r_v$, which realize the prescribed total geodesic curvatures.
	\item $\mathcal{E}$ is proper if and only if $\{\hat{L}_v\}_{v\in V}$ satisfies \eqref{condition}.
\end{enumerate}
\end{thm}  
\begin{proof} For the completeness of the paper, we include Nie's proof. The first statement follows from the computation of $\frac{\partial\mathcal{E}}{\partial K}$. For the second statement, we give a brief proof.

	We denote by $\mathcal{R}$ the set $(0,\frac{\pi}{2})^{N}$. Given the prescribed total geodesic curvatures $\{\hat{L}_v\}_{v\in V}$, we denote by $\mathcal{E}_0$ another strictly convex potential $\mathcal{E}+\sum_{v\in V}\hat{L}_vK_v$. We consider the map $W = (\nabla\mathcal{E}_0)\circ G$, where $G$ is a function which maps from $\mathcal{R}$ to $\mathbb{R}^{N}$ with
	\begin{align*}
	G(r_1,...,r_{N})=(\ln\cot r_1,...,\ln\cot r_{N}).
	\end{align*}
	Since $\mathcal{E}_0$ is strictly convex, the map $\nabla\mathcal{E}_0$ is injective. So that $W$ is also injective. As in \cite{MR2022715}, we call the domain $\Phi$ determined by \eqref{condition} in $\mathbb{R}^{N}$ a \enquote{coherent angle system}. Using the invariance of domain theorem we can prove that $W$ is an open map whose image is contained in $\Phi$. One can also prove that $W(\mathcal{R})$ is closed in $\Phi$, see arguments in \cite{guo2007note}. So the image of $W$ equals to $\Phi$ by the connectivity of $\Phi$. Therefore, if $\{\hat{L}_v\}_{v\in V}$ belongs to the coherent angle system, then there is a $K^* = (K^*_{v_1},...,K^*_{v_N}) $ such that $\nabla \mathcal{E}_0(K^*)=(\hat{L}_{v_1},...,\hat{L}_{v_N})$. By the definition of $\mathcal{E}_0$, $K^*$ is the critical point of $\mathcal{E}$, so that $\mathcal{E}$ is proper by strict convexity. This finishes the proof of the \enquote{if} part. The \enquote{only if} part can be derived with the help of \eqref{gauss}.
\end{proof} 

  A prescribed geodesic curvature flow \eqref{flow} can be changed into a negative gradient flow \eqref{2.2} of the strictly convex potential function $\mathcal{E}$. Therefore we can use a lemma in \cite{MR4334399}.
 
\begin{lem}\label{convex}
	Let $H$ be a strictly convex smooth function defined in a convex set $\Omega\subseteq \mathbb{R}^{N}$ with a critical point $P\in \Omega$. Then the following properties holds:
	\begin{enumerate}
		\item $P$ is the unique global minimum point of $H$.
		\item If $\Omega$ is unbounded, then $\lim_{|x|\rightarrow\infty}H(x)=\infty$.
	\end{enumerate}
\end{lem}
	We also recall a lemma in the theory of ordinary differential equations:
	\begin{lem}[\cite{MR0140742}]\label{convergent}
		Let $U$ be an open set in $\mathbb{R}^n$ and $f\in C^1(U,\mathbb{R}^n)$. Consider an autonomous ordinary differential system
		\begin{align}
			\ddt{x(t)}=f(x(t))~~~x(t)\in U.\label{ode}
		\end{align}
	Assume $x^*$ is a critical point of $f$, i.e. $f(x^*)=0$. If all eigenvalues of the Jacobian matrix $\frac{\partial f}{\partial x}(x^*)$ have negative real part, then $x^*$ is an asymptotically stable point. More specifically, there exists a neighborhood $\tilde{U}\subset U$ of $x^*$ such that, for any initial $x(0)\in \tilde{U}$, the solution of \eqref{ode} exists for all times $t\in [0,+\infty)$ and converges exponentially fast to $x^*$. 
	\end{lem}
Now we prove the main result.
\begin{proof}[Proof of Theorem \ref{thm1}]
	For \enquote{i$\Rightarrow$ii}, suppose that the flow \eqref{flow} converges to $r^*\in (0,\frac{\pi}{2})^{N}$ for some initial data. The flow \eqref{2.2} also converges to $K^*=\ln\cot r^*$. Picking a sequence of time $t_n=n\rightarrow\infty$, by the mean value theorem we have
	\begin{align}
		\mathcal{E}(K(n+1))-\mathcal{E}(K(n)) = -|\nabla \mathcal{E}(K(\eta_n))|^2 \text{~~~for some } \eta_n\in[n,n+1].\label{estimate}
	\end{align}
	Then by the convergence of $K(t)$ and the smoothness of $\mathcal{E}$, we have $\nabla\mathcal{E}(K^*)=0$. Therefore $K^*$ is the critical point of $\mathcal{E}$. Since the potential function $\mathcal{E}$ is strictly convex in $\mathbb{R}^{N}$, it is proper due to Lemma \ref{convex}. Then Theorem \ref{Nie} tells us the prescribed total geodesic curvatures $\{\hat{L}_v\}_{v\in V}$	satisfies \eqref{condition}.
	\\\hspace*{16pt}
	We prove \enquote{ii$\Rightarrow $i}. If the prescribed geodesic curvature satisfies \eqref{condition}, then by Theorem \ref{Nie}, the potential function $\mathcal{E}$ is proper. Therefore $\mathcal{E}$ has a critical point and is bounded from below. Since $\mathcal{E}$ decreases along the flow \eqref{2.2}, the flow $K(t)$ is contained in a compact subset of $\mathbb{R}^N$. Due to the monotonicity of $\mathcal{E}(K(t))$, $\{\mathcal{E}(K(n))\}_{n=1}^{\infty}$ converges to a finite value $\mathcal{E}_{\infty}$ as $n\rightarrow\infty$.

	By \eqref{estimate}, we have 
	\begin{align*}
		\lim_{n\rightarrow\infty} |\nabla \mathcal{E}(K(\eta_n))|^2=0.
	\end{align*}
     Since $\{K(\eta_n)\}_{n=1}^{\infty}$ is contained in a compact set, there exists a subsequence of $\{\eta_n\}_{n=1}^{\infty}$ denoted by
     $\{\eta_{n_k}\}_{k=1}^{\infty}$ and $K_{\infty}\in \mathbb{R}^N$ such that
     \begin{align*}
     	K(\eta_{n_k})\rightarrow K_{\infty}, ~~\text{as}~~k\rightarrow\infty.
     \end{align*}
 	 Therefore $\nabla\mathcal{E}(K_{\infty})=0$. Since $-\frac{\partial L}{\partial K}$ is negative-definite matrix by Lemma \ref{vari}, with the help of Lemma \ref{convergent}, $K(t)$ converges exponentially fast to $K_{\infty}$. Hence we finish the proof of Theorem \ref{thm1}.
\end{proof}
Finally, we give an example that some solution of the prescribed geodesic curvature problem could yield a smooth spherical metric.
\begin{ex}
Let $\Sigma$ be the sphere and $G=(V,E,F)$ be a triangulation of $\Sigma,$ isomorphic to the boundary of a tetrahedron. Assume that $\Theta(e)=\frac{\pi}{3}, \forall e\in E$ and $\hat{L}_v=\frac{2\pi}{3}, \forall v\in V.$
 The solution of the prescribed geodesic curvature problem is given by 
 $r_v=\arccos\frac{1}{3}, \forall v\in V.$ This produces a smooth metric structure $S(\Sigma),$ which is a standard spherical tiling of the sphere. 
 

\end{ex}

\textbf{Acknowledgements.} Ge is supported by NSFC, no.12122119. B. Hua is supported by NSFC, no.11831004, and by Shanghai Science and Technology Program [Project No. 22JC1400100].

	\section{Appendix}

 In this part, we use the setting in Section 1 and 2. Let $k_i,K_i,\beta_i,$$\theta$ denote $k_{v_i},K_{v_i},\frac{\angle Av_iB}{2}$,
 $\Theta(e_0)$ as in Figure \ref{fig02}. We give a detailed proof of Lemma \ref{vari}. We start with the cotangent 4-part formula for spherical triangles.
 \begin{prop}[Cotangent 4-part formula]
 	\begin{align}
 		\cot(\beta_1)=\frac{1}{\sin\theta}(\cot r_2\sin r_1+\cos r_1\cos\theta)\label{4.1}.
 	\end{align}
 \end{prop}
It is a useful formula in classical spherical geometry, and we omit its proof. We are ready to prove Lemma \ref{vari}. 
\begin{proof}[Proof of Lemma \ref{vari}]
	 Since the geodesic curvature at $\partial D_i$ is $k_i=\cot r_i$, we can rewrite \eqref{4.1} as
	\begin{align}
		\cot(\beta_1)=\frac{1}{\sin\theta}(k_2\sin r_1+\cos r_1\cos\theta)\label{4.2}.
	\end{align}
	By differentiating the above equation in $k_2$ and using the relation $l_1=2\beta_1\sin r_1$, we have

\begin{align*}
	-\frac{\partial\beta_1}{\partial k_2}\frac{1}{\sin^2\beta_1}=\frac{\sin r_1}{\sin\theta}.
\end{align*}
Thus we have
\begin{align*}
	\frac{\partial l_1}{\partial k_2}=\frac{-2\sin^2\beta_1\sin^2r_1}{\sin \theta}.
\end{align*}
By using the sine law in spherical trigonometry, i.e. $\frac{\sin r_1}{\sin\beta_2}=\frac{\sin r_2}{\sin\beta_1}$, we have $\frac{\partial l_1}{\partial k_2}=\frac{\partial l_2}{\partial k_1}$ and $\frac{\partial L_1}{\partial K_2}=\frac{\partial L_2}{\partial K_1}$. And we get
\begin{align*}
	\frac{\partial L_1}{\partial K_2}=k_1k_2\frac{-2\sin^2\beta_1\sin^2r_1}{\sin \theta}.
\end{align*}
Now we compute $\frac{\partial l_1}{\partial k_1}$. Again differentiating \eqref{4.2} in $k_1$, we have
\begin{align*}
	\frac{\partial l_1}{\partial k_1}=2\sin r_1[\frac{\sin^2\beta_1}{\sin\theta}(k_2\cos r_1\sin^2r_1-\cos\theta\sin^3r_1)-\frac{l_1}{2}\cos r_1].
\end{align*}
Noting that $\frac{\partial L_1}{\partial k_1}=l_1+k_1\frac{\partial l_1}{\partial k_1}$, we have
\begin{align*}
	\frac{\partial L_1}{\partial K_1}=k_1l_1\sin^2r_1+\frac{2\sin^2\beta_1}{\sin\theta}k_1k_2\sin^2r_1\cos^2r_1-2k_1\cot\theta\sin^2\beta_1\sin^3r_1\cos r_1.
\end{align*}
Using \eqref{4.1} again, we get 
\begin{align}
	\frac{\partial(L_1+L_2)}{\partial K_1} = \cos r_1\sin^2 r_1(2\beta_1-\sin 2\beta_1),
\end{align}
which is positive when $0<r_1<\frac{\pi}{2}$. Thus we finish the proof of Lemma \ref{vari}.
\end{proof}

\bibliographystyle{plain}
\bibliography{reference}

\begin{thebibliography}{10}

\bibitem{barnette1987generating}
David~W Barnette.
\newblock Generating closed 2-cell embeddings in the torus and the projective
  plane.
\newblock {\em Discrete \& computational geometry}, 2:233--247, 1987.

\bibitem{MR2022715}
Alexander~I. Bobenko and Boris~A. Springborn.
\newblock Variational principles for circle patterns and {K}oebe's theorem.
\newblock {\em Trans. Amer. Math. Soc.}, 356(2):659--689, 2004.

\bibitem{burago2022course}
Dmitri Burago, Yuri Burago, and Sergei Ivanov.
\newblock {\em A course in metric geometry}, volume~33.
\newblock American Mathematical Society, 2022.

\bibitem{MR2015261}
Bennett Chow and Feng Luo.
\newblock Combinatorial {R}icci flows on surfaces.
\newblock {\em J. Differential Geom.}, 63(1):97--129, 2003.

\bibitem{MR1106755}
Yves Colin~de Verdi\`ere.
\newblock Un principe variationnel pour les empilements de cercles.
\newblock {\em Invent. Math.}, 104(3):655--669, 1991.

\bibitem{coxeter1950self}
Harold~SM Coxeter.
\newblock Self-dual configurations and regular graphs.
\newblock {\em Bulletin of the American Mathematical Society}, 56(5):413--455,
  1950.

\bibitem{feng2022combinatorial2}
Ke~Feng, Huabin Ge, and Bobo Hua.
\newblock Combinatorial ricci flows and the hyperbolization of a class of
  compact 3--manifolds.
\newblock {\em Geometry \& Topology}, 26(3):1349--1384, 2022.

\bibitem{ge2018combinatorial}
Huabin Ge.
\newblock Combinatorial calabi flows on surfaces.
\newblock {\em Transactions of the American Mathematical Society},
  370(2):1377--1391, 2018.

\bibitem{MR3818085}
Huabin Ge and Bobo Hua.
\newblock On combinatorial {C}alabi flow with hyperbolic circle patterns.
\newblock {\em Adv. Math.}, 333:523--538, 2018.

\bibitem{ge20203}
Huabin Ge and Bobo Hua.
\newblock 3-dimensional combinatorial yamabe flow in hyperbolic background
  geometry.
\newblock {\em Transactions of the American Mathematical Society},
  373(7):5111--5140, 2020.

\bibitem{MR4334399}
Huabin Ge, Bobo Hua, and Ze~Zhou.
\newblock Circle patterns on surfaces of finite topological type.
\newblock {\em Amer. J. Math.}, 143(5):1397--1430, 2021.

\bibitem{ge2021combinatorial}
Huabin Ge, Bobo Hua, and Ze~Zhou.
\newblock Combinatorial ricci flows for ideal circle patterns.
\newblock {\em Advances in Mathematics}, 383:107698, 2021.

\bibitem{ge2022deformation}
Huabin Ge, Wenshuai Jiang, and Liangming Shen.
\newblock On the deformation of ball packings.
\newblock {\em Advances in Mathematics}, 398:108192, 2022.

\bibitem{MR2136535}
David Glickenstein.
\newblock A combinatorial {Y}amabe flow in three dimensions.
\newblock {\em Topology}, 44(4):791--808, 2005.

\bibitem{MR3825607}
Xianfeng Gu, Ren Guo, Feng Luo, Jian Sun, and Tianqi Wu.
\newblock A discrete uniformization theorem for polyhedral surfaces {II}.
\newblock {\em J. Differential Geom.}, 109(3):431--466, 2018.

\bibitem{MR3807319}
Xianfeng~David Gu, Feng Luo, Jian Sun, and Tianqi Wu.
\newblock A discrete uniformization theorem for polyhedral surfaces.
\newblock {\em J. Differential Geom.}, 109(2):223--256, 2018.

\bibitem{gu2008computational}
Xianfeng~David Gu and Shing-Tung Yau.
\newblock {\em Computational conformal geometry}, volume~1.
\newblock International Press Somerville, MA, 2008.

\bibitem{guo2007note}
Ren Guo.
\newblock A note on circle patterns on surfaces.
\newblock {\em Geometriae Dedicata}, 1(125):175--190, 2007.

\bibitem{MR2100762}
Feng Luo.
\newblock Combinatorial {Y}amabe flow on surfaces.
\newblock {\em Commun. Contemp. Math.}, 6(5):765--780, 2004.

\bibitem{MR3178441}
Feng Luo.
\newblock Rigidity of polyhedral surfaces, {I}.
\newblock {\em J. Differential Geom.}, 96(2):241--302, 2014.

\bibitem{luo2019koebe}
Feng Luo and Tianqi Wu.
\newblock Koebe conjecture and the weyl problem for convex surfaces in
  hyperbolic 3-space.
\newblock {\em arXiv preprint arXiv:1910.08001}, 2019.

\bibitem{nie2023circle}
Xin Nie.
\newblock On circle patterns and spherical conical metrics.
\newblock {\em arXiv preprint arXiv:2301.09585}, 2023.

\bibitem{MR0140742}
L.~S. Pontryagin.
\newblock {\em Ordinary differential equations}.
\newblock ADIWES International Series in Mathematics. Addison-Wesley Publishing
  Co., Inc., Reading, Mass.-Palo Alto, Calif.-London, 1962.
\newblock Translated from the Russian by Leonas Kacinskas and Walter B. Counts.

\bibitem{thurston1976geometry}
William~P. Thurston.
\newblock The geometry and topology of three-manifolds.
\newblock {\em Princeton lecture notes}, 1976.

\end{thebibliography}

\noindent Huabin Ge, hbge@ruc.edu.cn\\[2pt]
\emph{School of Mathematics, Renmin University of China, Beijing, 100872, P.R. China} 
\\

\noindent Bobo Hua, bobohua@fudan.edu.cn\\[2pt]
\emph{School of Mathematical Sciences, LMNS, Fudan University, Shanghai, 200433, P.R. China}
\\

\noindent Puchun Zhou, pczhou22@m.fudan.edu.cn\\[2pt]
\emph{School of Mathematical Sciences, Fudan University, Shanghai, 200433, P.R. China}
\end{document}